\documentclass[a4paper,dvipsnames]{article}

\usepackage{amsthm,amssymb,amsmath,enumerate,graphicx,epsf,bm,caption,framed,relsize}
\usepackage{authblk}
\usepackage[greek,english]{babel}  
\usepackage[percent]{overpic}
\usepackage{epstopdf}
\usepackage{bbm}

\newcommand{\COLORON}{0}
\newcommand{\NOTESON}{0}
\newcommand{\Debug}{0}
\usepackage[usenames]{color} 
\usepackage{amsthm,amssymb,amsmath,bbm,enumerate,graphicx,epsf,stmaryrd}
\usepackage[bookmarks, colorlinks=false, breaklinks=true]{hyperref} %REMOVE FOR ARXIV

\newcommand{\comment}[1]{}
\newcommand{\COMMENT}[1]{}

\definecolor{darkgray}{rgb}{0.3,0.3,0.3}
\newcommand{\defi}[1]{{\color{darkgray}\emph{#1}}}

\newcommand{\acknowledgement}{\section*{Acknowledgement}}

\comment{
	\begin{lemma}\label{}	
\end{lemma}
\begin{proof}

\end{proof}

\begin{theorem}\label{}
\end{theorem} 
\begin{proof} 	

\end{proof}

}

\newtheorem{proposition}{Proposition}[section]

\newtheorem{theorem}[proposition]{Theorem}
\newtheorem{corollary}[proposition]{Corollary}

\newtheorem{lemma}[proposition]{Lemma}

\newtheorem{examp}[proposition]{Example}%[section]

\newcommand{\FIG}{0}

\ifnum \NOTESON = 1 \newcommand{\note}[1]{ 

\hspace*{-30pt}
	{\color{blue}  NOTE: \color{Turquoise}{\small  \tt \begin{minipage}[c]{1.1\textwidth}  #1 \end{minipage} \ignorespacesafterend }} 
	
	}
\else \newcommand{\note}[1]{} \fi

\newcommand{\afsubm}[1]{ \ifnum \Debug = 1 {\mymargin{#1}}
\fi} %For notes on after-submission changes

\ifnum \Debug = 1 
\else  \fi

\ifnum \FIG = 1 \newcommand{\fig}[1]{Figure ``{#1}''}
\else \newcommand{\fig}[1]{Figure~\ref{#1}} \fi

\ifnum \FIG = 1 
\else  \fi

\ifnum \Debug = 1 \usepackage[notref,notcite]{showkeys}
\fi

\ifnum \COLORON = 0 \renewcommand{\color}[1]{}
\fi

\newcommand{\N}{\ensuremath{\mathbb N}}
\newcommand{\R}{\ensuremath{\mathbb R}}
\newcommand{\C}{\ensuremath{\mathbb C}}
\newcommand{\Z}{\ensuremath{\mathbb Z}}

\newcommand{\oo}{\ensuremath{\omega}}

\makeatletter
\DeclareRobustCommand{\cev}[1]{%
  \mathpalette\do@cev{#1}%
}
\newcommand{\do@cev}[2]{%
  \fix@cev{#1}{+}%
  \reflectbox{$\m@th#1\vec{\reflectbox{$\fix@cev{#1}{-}\m@th#1#2\fix@cev{#1}{+}$}}$}%
  \fix@cev{#1}{-}%
}
\newcommand{\fix@cev}[2]{%
  \ifx#1\displaystyle
    \mkern#23mu
  \else
    \ifx#1\textstyle
      \mkern#23mu
    \else
      \ifx#1\scriptstyle
        \mkern#22mu
      \else
        \mkern#22mu
      \fi
    \fi
  \fi
}

\makeatother
\newcommand{\seq}[1]{\ensuremath{(#1_n)_{n\in\N}}} 

\newcommand{\g}{\ensuremath{G\ }}

\newcommand{\Ex}{\mathbb E}
\renewcommand{\Pr}{\mathbb{P}}

\newcommand{\Lr}[1]{Lemma~\ref{#1}}

\newcommand{\Tr}[1]{Theorem~\ref{#1}}

\newcommand{\Sr}[1]{Section~\ref{#1}}

\newcommand{\fe}{for every}

\newcommand{\st}{such that}

\newcommand{\wrt}{with respect to}

\newcommand{\labtequ}[2]{%\labtequc{#1}{#2}}
 \begin{equation} \label{#1} 	\begin{minipage}[c]{0.9\textwidth}  #2 \end{minipage} \ignorespacesafterend \end{equation} }

\newcommand{\mymargin}[1]{% <- dieses % verhindert ein ungewolltes Leerzeichen
 \ifnum \Debug = 1
  \marginpar{%
    \begin{minipage}{\marginparwidth}\small%
      \begin{flushleft}%
        {\color{blue}#1}%
      \end{flushleft}%
   \end{minipage}%
  }%
 \fi
}%

\newcommand{\mySection}[2]{}

\usepackage{xcolor}

\usepackage{mathtools}

\RequirePackage{latexsym} \RequirePackage{amsthm}
\RequirePackage{amsmath}
\usepackage{hyperref}

\usepackage{enumerate}
\RequirePackage{amssymb}\RequirePackage{makeidx}
\usepackage{dsfont}
\usepackage{mathrsfs}
\newcommand{\myremark}[1]{\ifnum \Debug = 1 \tiny #1 \fi}

\newcommand{\sboxt}{\mathsmaller {\mathsmaller \boxtimes}}
\newcommand{\Ls}{\mathbb{L}^d_\sboxt}

\newcommand{\vx}{\ensuremath{\bm{x}}}

\newcommand{\pint}{interface}
\newcommand{\mpint}{multi-interface}

\usepackage{authblk}

\title{Analyticity of the percolation density $\theta$ in all dimensions}

\author[1]{Christoforos Panagiotis}
\author[2]{Agelos Georgakopoulos}
\affil[1,2]{{Mathematics Institute}\\
        {University of Warwick}\\
        {CV4 7AL, UK}\thanks{Supported by the European Research Council (ERC) under the European Union's Horizon 2020 research and innovation programme (grant agreement No 639046).}\\}
        
\begin{document}
\date{}
\maketitle

\begin{abstract}
We prove that for Bernoulli bond percolation on $\Z^d$, $d\geq 2$ the
percolation density is an analytic function of the parameter in the supercritical interval $(p_c,1]$. This answers a question of Kesten from 1981.
%For this, we show that for every $p>p_c$, whenever the cluster of the origin $C_o$ is finite, there exists a minimal edge cut of closed edges separating the origin from infinity which has an exponential tail. This is in contrast to the behaviour of the boundary of $C_o$ which has only a stretched exponential tail on the interval $(p_c,1-p_c]$ \cite{KeZhaPro,ExpGrowth}. As a corollary, we obtain an earlier result of Pete \cite{Pete} about the exponential decay of the probability that $C_o$ is finite but has a lot of boundary edges with the infinite component.
\end{abstract}

inclusion-exclusion.

\section{Introduction}

Perhaps the first occurrence of questions of smoothness in percolation theory dates back to the work of Sykes \& Essam \cite{SykesEssam}. Trying to compute the value of $p_c$ for bond percolation on the square lattice $\Z^2$, Sykes \& Essam obtained that the free energy (aka.\ mean number of clusters per vertex)  $\kappa(p):=\Ex_p(|C_o|^{-1})$, where $C_o$ denotes the cluster of the origin, satisfies the functional equation $\kappa(p)=\kappa(1-p)+\phi(p)$ for some polynomial $\phi(p)$. Under the assumption of smoothness of $\kappa$ for every value of the parameter $p$ other than $p_c$, at which it is conjectured that $\kappa$ has a singularity, they obtained that $p_c=1/2$ due to the symmetry of the functional equation around $1/2$. Their work generated considerable interest, and a lot of the early work in percolation was focused on the smoothness of functions like $\kappa$ and $\chi:= \mathbb{E}_p(|C_o|)$ that describe the macroscopic behaviour of its clusters. Kunz \& Souillard \cite{KunSou} proved that $\kappa$ is analytic for small enough $p$. Grimmett \cite{GriDif} proved that $\kappa$ is $C^\infty$ for $p\neq p_c$ in the case $d=2$. A breakthrough was made by Kesten \cite{Ke81}, who proved that $\kappa$ and $\chi$ are analytic on $[0,p_c)$ for all $d\geq 2$. (Despite all the efforts, the argument of Sykes \& Essam has never been made rigorous, and all proofs of the fact that $p_c=1/2$  when $d=2$ use different methods, see e.g.\ \cite{BoRioShor,KestenCritical}.)

Except for the special case of $\kappa$ on $\Z^2$ (and other planar lattices), smoothness results are harder to obtain in the supercritical interval $(p_c,1]$, partly because the cluster size distribution $P_n:=\Pr_p(|C(o)| = n)$ has an exponential tail below $p_c$ \cite{AB,MenCoi} but not above $p_c$ \cite{AiDeSoLow}. Still, it is known that $\kappa$ and $\theta:=\Pr(|C_o|=\infty)$ are infinitely  differentiable for $p\in (p_c,1]$ (see \cite{ChChNeBer} or \cite[\S 8.7]{Grimmett} and references therein). It was a well-known open question, dating back to \cite{Ke81} at least, and appearing in several textbooks (\cite[Problem 6]{KestenBook},\cite{GrimmettDisordered,Grimmett}), whether $\theta$ is analytic for $p\in (p_c,1]$ for percolation on the hypercubic lattice $\Z^d, d\geq 2$. This paper answers this question in the affirmative. %, we prove that $\theta$ is analytic above $p_c$ for all dimensions $d\geq 3$. 

Part of the interest for this question comes form Griffiths' \cite{Griffiths} discovery of models, constructed by applying the Ising model on 2-dimensional percolation clusters, in which the free energy is infinitely differentiable but not analytic. This phenomenon is since called a \defi{Griffiths singularity}, see \cite{EntGri} for an overview and further references. 

The study of the analytical properties of the free energy is a common theme in several models of Statistical Mechanics. Perhaps the most famous such example is Onsager's exact calculation of the free energy of the square-lattice Ising model \cite{Onsager}. A corollary of this calculation is the computation of the critical temperature, as well as the analyticity of the free energy for all temperatures other than the critical one. See also \cite{KLM} for an alternative proof of the latter result. The analytical properties of the free energy have also been studied for the $q$-Potts model, which generalizes the Ising model. For this model, the analyticity of the free energy has been proved for $d=2$ and all supercritical temperatures when $q$ is large enough \cite{Complete}. 

Before our result, partial progress on the analyticity of the percolation density had been made by Braga et.al.\ \cite{BrPrSaSco,BrPrSa}, who showed that $\theta$ is analytic for $p$ close enough to $1$. We recently settled the 2-dimensional case  \cite{analyticity} by introducing a notion of \defi{interfaces} that has already found further applications \cite{SitePercoPlane}. Shortly after our paper \cite{analyticity}  was released, Hermon and Hutchcroft \cite{HerHutSup} proved that $\theta$ is analytic above $p_c$ for every non-amenable transitive graph, by establishing that the cluster size distribution $P_n$ has an exponential tail in the whole supercritical interval.

\medskip
Our proof of the analyticity of $\theta(p)$  on the supercritical interval involves expressing the function as an infinite sum of polynomials $f_n(p)$, and then extending $p$ to the complex plane. To show that this sum converges to an analytic function, we need suitable upper bounds for $|f_n(z)|$ inside regions of the complex plane. These bounds can be obtained once $f_n$ decays to $0$ fast enough. Possible candidates for $f_n$ are the probabilities $P_n$, since one can write $\theta(p)=1-\sum_{n} P_n(p)$. However, as $P_n$ decays slower than exponentially for $p>p_c$ \cite{Grimmett,KunSou}, the bounds we obtain for $|P_n(z)|$ do not provide the desired convergence. Instead of working with the whole of $C_o$, an alternative approach is to work with the `perimeter' of its boundary. As we observed in \cite{analyticity}, in the planar case, the suitable notion of perimeter turns out to be the \defi{interface} of $C_o$. An interface consists of a set of closed edges that we call the \defi{boundary} of the interface, and separate $C_o$ from infinity, and   a set of open edges that is part of $C_o$ and incident to the boundary (\fig{figscv}). With this definition we obtain that $1-\theta(p)$ coincides with the probability $\Pr_p(\text{at least one \pint\ occurs})$, which can be expanded as a sum over all \pint s, i.e.\ over all subgraphs of the lattice that could potentially coincide with the interface of $C_o$. Since several \pint s might occur simultaneously, we have to apply the inclusion-exclusion principle. Thus we obtain
$$1-\theta(p)=\sum (-1)^{c(M)+1}\Pr_p(\text{M occurs}),$$
where the sum ranges over all finite collections of edge-disjoint \pint s, called \defi{\mpint s}, and $c(M)$ denotes the number of \pint s of the collection. 

\begin{figure}[htbp]
\vspace*{5mm}
\centering
\noindent

\begin{overpic}[width=.6\linewidth]{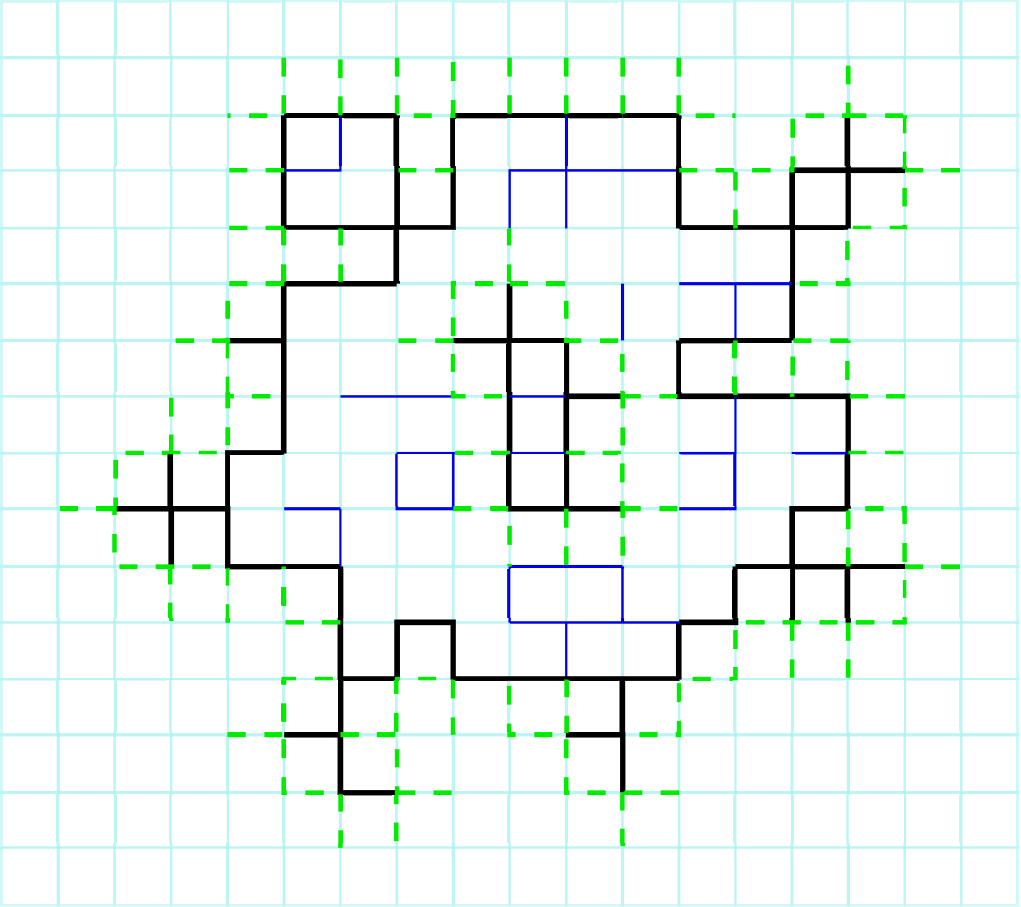}
% sphere generated from matlab % sphere; cpmmand
%\centering \begin{overpic}[width=.6\linewidth]{regions} 
\put(51,40){$o$}
\end{overpic}
\begin{minipage}[c]{0,9\textwidth}
\vspace*{5mm}
\caption{\small An example of two nested interfaces, depicted in dark solid lines. Dashed lines depict the boundary of the interface. The rest of the clusters is depicted in plain lines (blue, if colour is shown).} \label{figscv}
\end{minipage}
\end{figure}

For any plausible \mpint\ $M$, the probability $\Pr_p(\text{M occurs})$ is just $P_M(p):=p^{|M|}(1-p)^{|\partial M|}$ by the definitions, where $|M|$ and $|\partial M|$ denote the number of edges of the \mpint\ and its boundary, respectively (in \fig{figscv}, these are depicted in dark solid lines and dashed lines,  respectively). We can extend this polynomial expression to $\C$ hoping to obtain strong enough upper  bounds for $|P_M(z)|$. In the special case where $M$ comprises a single \pint, these bounds are obtained by combining the well-known coupling between supercritical bond percolation on $\Z^2$ and subcritical bond percolation on its dual with the exponential tail of $P_n$ on the subcritical interval. In the general case, the bounds are obtained by some combinatorial arguments and the BK inequality. See  \cite{analyticity} for details.

Our notion of interfaces can be generalised to higher dimensions in such a way that a unique interface is associated to any cluster. A slight modification of the above method still yields the analyticity of $\theta$ for the values of $p$ close to $1$, but not in the whole supercritical interval. The main obstacle is that for values of $p$ in the interval $(p_c,1-p_c)$, the distribution of the size of the interface of $C_o$ has only a stretched exponential tail, which follows from the work of Kesten and Zhang \cite{KeZhaPro}. (As we observed in \cite{ExpGrowth}, this behaviour holds for $p=1-p_c$ as well.)

In the same paper \cite{KeZhaPro}, Kesten and Zhang introduced some variants of the standard boundary of $C_o$ that are obtained by dividing the lattice $\Z^d$ into large boxes, and proved that these variants satisfy the desired exponential tail on the whole supercritical interval.\footnote{The threshold $p_c(H^d)$ in Kesten's and Zhang's original formulation was proved later to coincide with $p_c(\Z^d)$ by Grimmett and Marstrand \cite{GriMar}.} It is natural to try to apply our method to those variants, however, their occurrence does not prevent the origin from being connected to infinity. Instead, we expand these variants into larger objects that we call \defi{separating components}. In Section \ref{theta} (\Lr{Co finite}) we prove that whenever a separating component $S$ occurs, we can find inside $S$ and its boundary $\partial_\sboxt S$ an edge cut  $\partial^b \mathcal{S}_o$ separating the origin from infinity. Conversely, some separating component occurs whenever $C_o$ is finite (\Lr{an S occurs}). Thus we can express $\theta$ in terms of the occurrence of separating components (see \eqref{cases} in \Sr{theta sum S}). 
In contrast to the behaviour of the boundary of $C_o$ which has only a stretched exponential tail on the interval $(p_c,1-p_c]$, $\partial^b \mathcal{S}_o$ has an exponential tail in the whole supercritical interval. We plug this exponential decay into a general tool from \cite{analyticity} (\Tr{cor general}), which rests on an application of the Weierstrass M-test to polynomials of the form $p^m(1-p)^n$, to obtain the analyticity of $\theta$ above $p_c$ in \Sr{expand}. In Section \ref{section tau} we use similar arguments to prove the analyticity of the $k$-point function $\tau$ and its truncation $\tau^f$, as well as of $\chi^f:= \mathbb{E}(|C(o)|; |C(o)|<\infty)$ and $\kappa$.

\medskip
Typically, $\partial^b \mathcal{S}_o$ has size of smaller magnitude than the boundary of $C_o$, and it is obtained from the latter by `smoothening' some of its parts with `fractal' structure. As a corollary, we re-obtain, in \Sr{Pete}, a result of Pete \cite{Pete} about the exponential decay of the probability that $C_o$ is finite but sends a lot of closed edges to the infinite component. 

\section{Preliminaries}

\subsection{Graph theory} \label{sec GT}

Consider an infinite connected graph $G=(V,E)$. Given a finite subgraph $H$ of $G$, we define its \defi{internal boundary} $\partial H$ to be the set of vertices of $H$ that are incident with an infinite component of $G\setminus H$.
We define the \defi{vertex boundary} $\partial^V H$ of $H$ as the set of vertices in $V\setminus V(H)$ that have a neighbour in $H$. The \defi{edge boundary} $\partial^E H$ is the set of edges in $E\setminus E(H)$ that are incident to $H$. 

Consider now a vertex $x$ of $G$. We say that a set $S$ of edges of $G$ is an \defi{edge cut} of $x$ if $x$ belongs to a finite component of $G-S$. We say that $S$ is a \defi{minimal edge cut} of $x$ if it is minimal with respect to inclusion. For a finite connected subgraph $H$ of $G$, its minimal edge cut is the set of edges with one endvertex in $H$ and one in an infinite component of $G\setminus H$.

The \defi{diameter} $diam(H)$ of $H$ is defined as $\max_{x,y\in V(H)}\{d_G(x,y)\}$ where $d_G(x,y)$ denotes the graph-theoretic distance between $x$ and $y$.

\subsection{Percolation}

We recall some standard definitions of percolation theory in order to fix our notation. For more details the reader can consult e.g.\ \cite{Grimmett,LyonsBook}.

Consider the hypercubic lattice $\mathbb{L}^d=(\Z^d,E(\Z^d))$, the vertices of which are the points in $\R^d$ with integer coordinates, and two vertices are connected with an edge when they have distance $1$. We let $\Omega:= \{0,1\}^{E(\Z^d)}$ be the set of \defi{percolation configurations} on $\mathbb{L}^d$. We say that an edge $e$ is \defi{closed} (respectively, \defi{open}) in a percolation configuration $\oo\in \Omega$, if $\oo(e)=0$ (resp.\ $\oo(e)=1$).

By Bernoulli, bond \defi{percolation} on $\mathbb{L}^d$ with parameter $p\in [0,1]$ we mean the random subgraph of $\mathbb{L}_d$ obtained by keeping each edge with probability $p$ and deleting it with probability $1-p$, with these decisions being independent of each other. The corresponding probability measure on the configurations of open and closed edges is denoted by $\Pr_p$. We also denote by $\Ex_p$ the expectation with respect to $\Pr_p$.

The \defi{percolation threshold} $p_c(\mathbb{L}^d)$ is defined by 
$$p_c(\mathbb{L}^d):= \sup\{p \mid \Pr_p(|C_o|=\infty)=0\},$$
where $o$ denotes the origin $(0,\ldots,0)\in \Z^d$, and its \defi{cluster} $C_o$  is the component of $o$ in the subgraph of $\mathbb{L}^d$ spanned by the open edges. We will write $C_o(\omega)$ when we want to emphasize the dependence of the cluster on the (random) percolation instance $\omega$.

\subsection{Analyticity}

In order to prove that $\theta$ and the other functions describing the macroscopic behaviour of our model are analytic we will utilize some results proved in \cite{analyticity}. The first result provides sufficient conditions for analyticity. The second result will be used when estimates for the analytic extensions of those functions are needed. 

We say that an event $E$ has \defi{complexity} $k$, if it is a disjoint union of a family of events $\seq{F}$ where each $F_n$ is measurable \wrt\ a set of edges of \g of cardinality at most $k$.

\begin{theorem}[\cite{analyticity}]\label{cor general}
Let $I\subset [0,1]$ be an interval and $f(p): I \to \R$ a function that can be expressed as a sum
$$f(p)=\sum_{n\in\mathbb{N}}\sum_{i\in L_n} a_i \Pr_p(E_{n,i})$$ 
where $a_n\in \mathbb{R}$, $L_n$ is a finite index set, and each $E_{n,i}$ is an event measurable with respect to
$\Pr_p$ (in particular, the above sum converges absolutely for every $p\in I$).
Suppose that
\begin{enumerate}
\item \label{cg i} each $E_{n,i}$ has complexity of order $\Theta(n)$, and 
\item \label{cg ii} for each open subinterval $J\subset I$ there is a constant $0<c_J<1$ such that
$\sum_{i\in L_n} a_i \Pr_p(E_{n,i})=O({c_J}^n)$.
\end{enumerate} 
Then $f(p)$ is analytic in $I$.
\end{theorem}

In the following lemma $D(p,\delta)$ denotes the open disk of radius $\delta$ centred at $p$.

\begin{lemma}[\cite{analyticity}] \label{C equals S NN}
For every finite subgraph $S$ of $G$ and every $o\in V(G)$, the function $P(p) :=\Pr_p(C_o=S)$ admits an entire extension $P(z), z\in \C$, \st\ \fe\ $0<\delta<1$, every $0\leq p<1$ with $p+\delta<1$ and every $z\in D(p,\delta)$, we have
$$|P(z)| \leq c^{|\partial^E S|} P(p+\delta),$$ where $c=c_{\delta,p}:= \frac{1-p+\delta}{1-p-\delta}$. Moreover, $|P(z)| \leq c_\delta^{|\partial^E S|} P(1-\delta)$ \fe\  $z\in D(1,\delta)$, where $c_\delta:= \frac{1+\delta}{1-\delta}$.
\end{lemma}

\section{Analyticity of $\theta$}\label{theta}

In this section we will prove that $\theta$ is analytic on the supercritical interval for every $d\geq 3$. (The case $d= 2$ was handled in \cite{analyticity}.)

\begin{theorem}\label{analytic}
For Bernoulli bond percolation on $\mathbb{L}^d , d\geq 3$, the percolation density $\theta(p)$ is analytic on $(p_c,1]$.
\end{theorem} 

\subsection{Setting up the renormalisation} \label{sec boxes}

We start by introducing some necessary definitions. Consider a positive integer $N$. For every vertex $x$ of $\Z^d$, we let $B(x)=B(x,N)$ denote the box $\{y\in \Z^d: \lVert y-Nx \rVert_{\infty}\leq 3N/4\}$. With a slight abuse, we will use the same notation $B(x)$ to also denote the corresponding subset of $\R^d$, namely $\{y\in \R^d: \lVert y-Nx \rVert_{\infty}\leq 3N/4\}$.

The collection of all these boxes can be thought of as the vertex set of graph canonically isomorphic to $\Z^d$. We will denote this graph by $N\mathbb{L}^d$. Whenever we talk about  percolation (clusters) from now on, we will be referring to percolation, with a fixed parameter $p>p_c$, on $\mathbb{L}^d$ and not on $N\mathbb{L}^d$; we will never percolate the latter. 

For any  percolation cluster $C$, %of diameter at least $N/5$, 
we denote by $C(N)$ the set of boxes $B$ such that the subgraph of $C$ induced by its vertices lying in $B$ has a component of diameter at least $N/5$. The boxes with this property will be called \defi{C-substantial}. Notice that $C(N)$ is a connected subgraph of $N\mathbb{L}^d$.
The internal boundary of $C(N)$ is denoted by $\partial C(N)$ following the terminology of \Sr{sec GT}. Notice that $\partial C(N)$ is not necessarily connected. For technical reasons, we would like it to be, and therefore we modify our lattice by adding the diagonals: we introduce a new graph $N\mathbb{L}^d_\sboxt$, the vertices of which are the boxes $B(x), x\in \Z^d$, and we connect two boxes with an edge of $N\mathbb{L}^d_\sboxt$ whenever they have non-empty intersection. When $N=1$, the vertex set of $\mathbb{L}^d_\sboxt$ is simply $\Z^d$. It is not too hard to show (see  \cite[Theorem 5.1]{TimCut}) that 
\labtequ{timar}{If $C$ is finite then $\partial C(N)$ is a connected subgraph of $N\mathbb{L}^d_\sboxt$.}

Given two diagonally opposite neighbours $x$, $y$ of $\mathbb{L}^d$, we will write $B(x,y)$ for the intersection $B(x) \cap B(y)$. A percolation cluster $C$ is a \defi{crossing cluster}  for some box $B(x)$ or $B(x,y)$, if $C$ contains a vertex from each of the $(d-1)$-dimensional faces of that box. We say that a box $B(x)$ is \defi{good} in a percolation configuration $\omega$ if it has a crossing cluster $C$ with the property that the intersection of $C$ with each of the boxes $B(x,y)$ contains a crossing cluster (of $B(x,y)$), and every other cluster of $B(x)$ has diameter less than $N/5$. A box that is not good will be called \defi{bad}. It is known \cite[Theorem 7.61]{Grimmett} that, for every $p>p_c$, the probability of having a crossing cluster and no other cluster of diameter greater than $N/5$ converges to $1$ as $N\to \infty$. Combining this with a union bound we easily deduce that 
\labtequ{good to 1}{for every $p>p_c$, the probability of any box being good converges to $1$ as $N\to \infty$.} 
We will say that a set of boxes is bad if all its boxes are bad.

Our definition of good boxes is slightly different than the standard one in that it asks for all boxes $B(x,y)$ to contain a crossing cluster. The reason for imposing  this additional property is because now 
\labtequ{star}{every $N\mathbb{L}^d_\sboxt$-component $B$ of good boxes contains a unique percolation cluster $C$ such that some box of $B$ is $C$-substantial (and in fact all boxes of $B$ are $C$-substantial).}  
This follows easily once we notice that this holds for pairs of neighbouring boxes.

Observe that the boxes in $\partial C(N)$ are never good. Indeed, if some box $B\in \partial C(N)$ is good, then $C$ connects all the $(d-1)$-dimensional faces of $B$, hence all $N\mathbb{L}^d$-neighbouring boxes of $B$ contain a connected subgraph of $C$ of diameter at least $N/5$, and so they lie in $C(N)$. This contradicts the fact that $B$ belongs to $\partial C(N)$. 

Having introduced the above definitions, our aim now is to find a suitable expression for $1-\theta$ in terms of good and bad boxes surrounding $o$. 

%Let us assume that the component of the origin $o$ is finite and has diameter at least $N/5$ \mymargin{I am assuming this so that $\partial C_o(N)$ is non-empty. In \ref{cases} we consider both cases}. 

With the above definitions we have that, conditioning on the event that $C_o$ is finite and has diameter at least $N/5$, there is a non-empty $N\Ls$-connected subgraph of bad boxes that separates $o$ from infinity, namely $T:=\partial C_o(N)$. However, the event $\{|C_o|<\infty\}$ is not necessarily measurable with respect to the configuration inside $T$. In other words, we cannot express $1-\theta$ in terms of just the configuration inside $T$, and instead we have to explore the configuration inside the finite components surrounded by $T$. To this end, we will expand $\partial C_o(N)$ into a larger object. 

\subsection{Separating components} \label{sec SCs}

A \defi{separating component} is a $N\mathbb{L}^d_\sboxt$-connected set $S$ of boxes, such that $o$ lies either inside $S$ or in a finite component of $N\mathbb{L}^d_\sboxt\setminus S$. We will write $\partial_\sboxt S$ for its vertex boundary ---defined in \Sr{sec GT}--- when viewed as a subgraph of $N\mathbb{L}^d_\sboxt$. We say that $S$ \defi{occurs} in a configuration $\omega$ if all the following hold:
\begin{enumerate}
\item \label{occ i} all boxes in $S$ are bad; 
\item \label{occ ii} all boxes in $\partial_\sboxt S$ are good, and
\item \label{occ iii} there is a configuration $\omega'$ which coincides with $\omega$ in $S\cup \partial_\sboxt S$, such that $C_o(\omega')$ is finite, and $S$ contains $\partial C_o(\omega')(N)$.
\end{enumerate}
We will say that $\omega'$ is a \defi{witness} for the occurrence of $S$ if \ref{occ i}--\ref{occ iii} all hold. 

One way to interpret \ref{occ iii} is that there exists a minimal cut set $F$ surrounding $o$ with the property that all its edges inside $S\cup \partial_\sboxt S$ are closed in $\omega$. If there is an infinite path in $\omega$ starting from $o$, then it has to avoid the edges of $F$ lying in $S\cup \partial_\sboxt S$. As we will see, \ref{occ ii} makes this impossible without violating that $C_o(\omega')$ is finite.

Note that \ref{occ iii} implies that 
\labtequ{S omega}{$\partial^V C_o(\omega')$ (and $C_o(\omega')$) does not share a vertex with the infinite component of $\mathbb{L}^d \setminus  S$.}

\subsection{Expressing $\theta$ in terms of the probability of the occurrence of a separating component} \label{theta sum S}

In this section we show that $C_o$ is finite exactly when some separating component occurs, unless $diam(C_o)< N/5$ which is a case that is easy to deal with. This will allow us to express $\theta(p)$ in terms of the probability of the occurrence of a separating component (see \eqref{cases}). In the following section we will expand the latter as a sum (with inclusion-exclusion) over all possible separating components. The summands of this sum are well-behaved polynomials, that will allow us to apply \Tr{cor general} to deduce the analyticity of $\theta(p)$.

\begin{lemma} \label{an S occurs}
For every $p>p_c$ there is $N\in \N$ and an interval $(a,b)$ containing $p$ such that the following holds for every $q\in (a,b)\cap (p_c,1]$. Conditioning on $C_o$ being finite, and $diam(C_o)\geq N/5$, at least one separating component occurs almost surely.
\end{lemma}
\begin{proof}
Let $S$ be the maximal connected subgraph of $N\mathbb{L}^d_\sboxt$ that contains $\partial C_o(N)$ and consists of bad boxes only. This $S$ exists whenever $C_o$ is finite and $diam(C_o)\geq N/5$ because $\partial C_o(N)$ is connected by \eqref{timar}.

We claim that there is some $N$ and an interval $(a,b)$ containing $p$ such that $S$ is $\Pr_q$-almost surely finite for every $q\in (a,b)\cap (p_c,1]$. For this, it suffices to show that for some large enough $N$, the probability $\Pr_q(\text{S has size at least $n$})$ converges to $0$ as $n$ tends to infinity for each such $q$. The latter follows by combining the union bound with \Lr{exp dec} below, which states that 
$$\sum_{T \text{ is a separating component of size } n} \Pr_q(\text{T is bad})\leq e^{-tn}$$ 
for some constant $t=t(p)>0$, for some $N$, and every $q$ in an interval $(a,b)\cap(p_c,1]$.

Note that conditions \ref{occ i} and \ref{occ ii} are automatically satisfied by the choice of $S$. The configuration $\omega':=\omega$ satisfies condition \ref{occ iii}, since $C_o(\omega)$ is finite, and $S$ contains $\partial C_o(\omega)(N)$ by definition. Thus $S$ occurs in $\omega$, as desired.
\end{proof}

%given a percolation instance  $\omega$ on \Lsat, we obtain a site percolation instance $\beta(\omega)$ on $\Ls$ by recording which boxes are good and which are bad: we let $\beta(\omega)(x)= 0$ if $B(x)$ is good and $\beta(\omega)(x)= 1$ if $B(x)$ is bad for each $x\in \mathbb{Z}^d$. Note that whether a box $B$ is good or bad depends on the state of the edges inside $B$ only. Since two boxes $B, B'$ overlap only when they are connected by an edge of $\Ls$, we can think of $\beta(\omega)$ as a realisation of an 1-dependent site percolation model on $\Ls$. By \eqref{good to 1} and \Tr{dependent}, $\beta(\omega)$ is subcritical when $N=N(p)$ is chosen large enough and $\omega$ is sampled from the law of $\Pr_p$. Since $S$ is a cluster of open (i.e.\ bad) sites of $\beta(\omega)$ by the definitions, it is almost surely  finite as claimed.

%Note that conditions \ref{occ i} and \ref{occ ii} are automatically satisfied by the choice of $S$. For \ref{occ iii}, let $\omega'$ be ... \mymargin{continue...}

%---------\\
%Is this true? Cite/prove \& move to a new preliminary section ...
%\begin{theorem}[] \label{dependent}
%For any 1-dependent site percolation model on $\Ls$ we have, $p_c>0$. 
%\end{theorem}
%---------

Note that the proof of \Lr{an S occurs} finds a concrete occurring separating component whenever $C_o$ is finite and $diam(C_o)\geq N/5$; we denote this separating component by $\mathcal{S}_o$ in this case.

\medskip
The next two lemmas provide a converse to \Lr{an S occurs}, namely that $C_o$ is finite whenever 
some separating component occurs. 

Whenever $\omega'$ is a witness for the occurrence of $S$, we let $R_o(\omega')$ denote the set of vertices of the infinite component of $\mathbb{L}^d\setminus C_o(\omega')$ lying in $S$. 

\begin{lemma}\label{finite}
Consider a separating component $S$, and assume that $S$ occurs in $\omega$. Let $\omega'$ be a witness of the occurrence of $S$. Then no vertex of $R_o(\omega')$ lies in $C_o(\omega)$.
\end{lemma}
\begin{proof}
Assume that some vertex $u$ of $R_o(\omega')$ lies in $C_o(\omega)$; we will obtain a contradiction.  

Since $C_o(\omega)$ contains $u$, there must exist a path $P$ in $\omega$ connecting $o$ to $u$. This path cannot lie entirely in $S\cup\partial_\sboxt S$ because $\omega$ and $\omega'$ coincide in that set of boxes and $u\not\in C_o(\omega')$. Hence $N\Ls\setminus (S\cup\partial_\sboxt S)$ must have some finite component. Let $E$ denote the minimal edge cut of $C_o(\omega')$. Clearly, $P$ must intersect $E$, since $u$ lies in the infinite component of $\mathbb{L}^d\setminus C_o(\omega')$. Let $e$ be an edge of $E$ that $P$ contains. Notice that no common edge of $P$ and $E$ lies in $S\cup \partial_\sboxt S$, because the edges of $E$ are closed in $\omega'$, the edges of $P$ are open in $\omega$, and the two configurations coincide in $S\cup \partial_\sboxt S$. Hence $e$ must lie in one of the finite components $\mathcal{B}_{in}$ of $N\Ls\setminus (S\cup\partial_\sboxt S)$. Write $\mathcal{B}$ for the set of those boxes in $\partial_\sboxt S$ that have a $N\mathbb{L}^d_\sboxt$-neighbour in $\mathcal{B}_{in}$. (Thus $\mathcal{B}$ is the vertex boundary of $\mathcal{B}_{in}$.) See Figure \ref{drawing}.

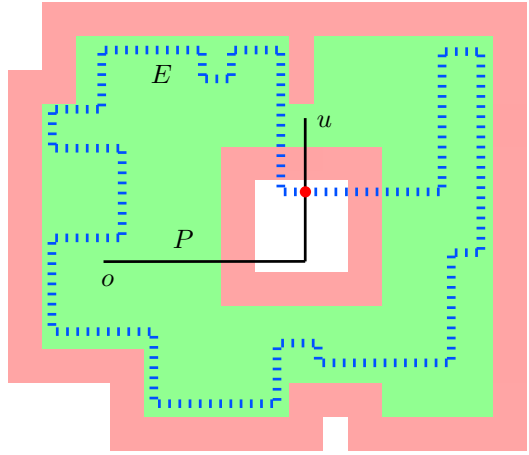
\begin{figure}[!ht] 
   \centering
   \noindent

\begin{overpic}[width=.6\linewidth]{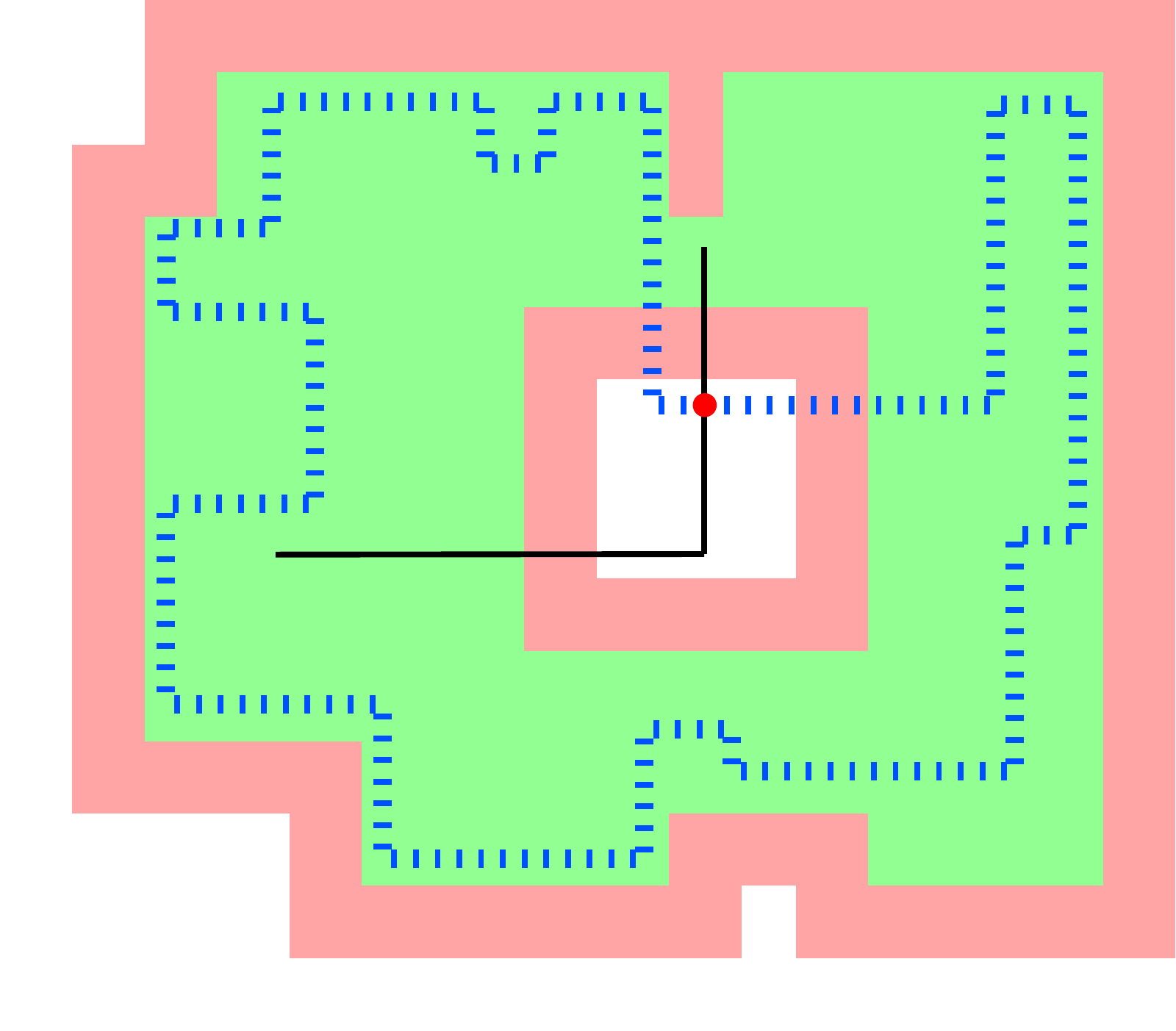} 
\put(23,36){$o$}
\put(62,65){$u$}
\put(32,73){$E$}
\put(36,43){$P$}
\end{overpic}
\caption{The situation in the proof of \Lr{finite}. The separating component $S$ is depicted in green and its boundary $\partial_\sboxt S$ in red (if colour is shown). When two boxes of $S$ and $\partial_\sboxt S$ overlap, their intersection is depicted also in green. The dashes depict the edges of the cut $E$, and $e$ is highlighted with a (red) dot.}\label{drawing}
\end{figure}

It is not hard to see that some box $B$ of $\mathcal{B}$ is $C_o(\omega')$-substantial, which then implies that all boxes of $\mathcal{B}$ are $C_o(\omega')$-substantial because they are all good. Indeed, notice that one of the two endvertices of $e$ lies in $C_o(\omega')$ by the definition of the set $E$. As $S$ contains a  $C_o(\omega')$-substantial box, some box $B$ of $\mathcal{B}$ must be $C_o(\omega')$-substantial, as claimed, because $\mathcal{B}$ is the vertex boundary of $\mathcal{B}_{in}$. 

Our aim now is to show that we can connect $u$ to the subgraph of $C_o(\omega')$ inside $\mathcal{B}$ with a path in $\omega'$ lying entirely in $S\cup\partial_\sboxt S$. This will imply that $u$ belongs to $C_o(\omega')$, contradicting that $u\in R_o(\omega')$.

For this, consider the subpath $Q$ of $P$ that starts at $u$ and ends at the last vertex of the intersection of $\mathcal{B}_{in}$ and $\mathcal{B}$ (notice that although $\mathcal{B}_{in}$ and $\mathcal{B}$ are disjoint sets of boxes, the subgraphs of $\mathbb{L}^d$ inside them overlap). If $Q$ is not contained in $S\cup\partial_\sboxt S$, then we can modify it to ensure that it does lie entirely in $S\cup\partial_\sboxt S$. Indeed, notice that each $N\mathbb{L}^d_\sboxt$-component $F$ of $\partial_\sboxt S$ contains a unique  $\omega$-cluster $C$ such that some box of $F$ is $C$-substantial by \eqref{star}, because all its boxes are good. Moreover, each time $Q$ exits $S\cup \partial_\sboxt S$, it has to first visit the unique such percolation cluster of some $N\mathbb{L}^d_\sboxt$-component $F$ of $\partial_\sboxt S$, and then eventually revisit the same percolation cluster of $F$. We can thus replace the subpaths of $Q$ that lie outside of $S\cup\partial_\sboxt S$ by open paths lying entirely in $\partial_\sboxt S$ that share the same endvertices. Thus we may assume that $Q$ is contained in $S\cup\partial_\sboxt S$ as claimed.

Now notice that $Q$ contains a subpath of diameter greater than $N/5$ lying entirely in some box $B$ of $\mathcal{B}$. This box is $C_o(\omega')$-substantial, hence $C_o(\omega')$ and $Q$ must meet. Then following the edges of $Q$, which are all open in $\omega'$, we arrive at $u$, and thus $u$ belongs $C_o(\omega')$, as desired. 
\end{proof}

We now use this to prove

\begin{lemma}\label{Co finite}
Whenever some separating component occurs in a configuration $\omega$, the cluster $C_o(\omega)$ is finite.
\end{lemma}
\begin{proof}
We will prove the following slightly stronger statement: whenever a separating component $S$ occurs in a configuration $\omega$, a minimal (finite) edge cut of closed edges occurs in $\omega$ which separates $o$ from infinity and lies in $S\cup \partial_\sboxt S$. 

For this, consider a witness $\omega'$ of the occurrence of $S$, and let $\omega''$ be the configuration which coincides with $\omega$ (and $\omega'$) on every edge lying in $S\cup\partial_\sboxt S$, and every other edge of $\omega''$ is open.  Note that $S$ occurs in $\omega''$ since it occurs in $\omega$. Thus $C_o(\omega'')$ contains no vertex of $R_o(\omega')$ by \Lr{finite}. This implies that $C_o(\omega'')$ contains no vertex in the infinite component $X$ of $N\Ls\setminus S$, because any path $P$ in $\mathbb{L}$ connecting $o$ to $X$ has to first visit $R_o(\omega')$. To see that the latter statement is true, consider the last vertex $u$ of $\partial^V C_o(\omega')$ that $P$ contains. Notice that the subpath of $P$ after $u$, which is denoted $Q$, visits only vertices of the infinite component of $\mathbb{L}^d\setminus C_o(\omega')$, and furthermore that $u$ lies either in $S$ or in a finite component of $\mathbb{L}^d\setminus S$ by \eqref{S omega}. In the first case, $u$ lies in $R_o(\omega')$. In the second case, $Q$ has to visit $S$, hence $R_o(\omega')$.

We have just proved that $C_o(\omega'')$ can only contain vertices in $S$ and the finite components of $N \Ls\setminus S$. Since $S$ is a finite set of boxes, $C_o(\omega'')$ is finite as well. Hence a minimal edge cut of closed edges separating $o$ from infinity occurs in $\omega''$. This minimal edge cut must lie entirely in $S\cup \partial_\sboxt S$, because all edges not in $S\cup \partial_\sboxt S$ are open. This is the desired minimal edge cut since it occurs in $\omega$ as well. We will %call it the \defi{bad edge cut} of the separating component $\mathcal{S}$ of $C_o(\omega)$, and we will 
denote it by $\partial^b \mathcal{S}_o$.
\end{proof}

Lemmas~\ref{an S occurs} and~\ref{Co finite} combined allow us to express the event that $C_o$ is finite in terms of the event that some separating component occurs. To do so, let us write $D_N$ to denote the event $\{diam(C_o)<N/5\}$.
Thus we have proved that  
\begin{align}\label{cases}
\begin{split}
1-\theta(p)=&\Pr_p(C_o \text{ is finite})\\ = &\Pr_p(D_N)+\Pr_p(|C_o|<\infty, D^\mathsf{c}_N)\\ =&\Pr_p(D_N)+\Pr_p(\text{some separating component occurs},D^\mathsf{c}_N).
\end{split}
\end{align}
Here and below, the notation $X,Y,\ldots$ denotes the intersection of the events $X,Y,\ldots$. 

\subsection{Expanding $\theta$ as an infinite sum of polynomials} \label{expand}

Notice that $\Pr_p(D_N)$ is a polynomial in  $p$, since the event $D_N$ depends only on the state of finitely many edges. 

Following our technique from \cite{analyticity} we will now use the inclusion-exclusion principle to expand the right-hand side $\Pr_p(\text{some  separating component occurs},D^\mathsf{c}_N)$  of \eqref{cases} as an infinite sum of polynomials, corresponding to all possible separating components that could occur.

Notice that any two occurring separating components are disjoint because they are connected, their boxes are bad, and they are surrounded by good boxes by definition.

\begin{lemma}\label{inc-ex lemma}
For every $p>p_c$ there is some integer $N=N(p)>0$ and an interval $(a,b)$ containing $p$ such that the expansion
\begin{align}\label{inc-ex}
\Pr_q(\text{some S occurs},D^\mathsf{c}_N)=\sum_{S\in MS^N} (-1)^{c(S)+1} \Pr_q(\text{S occurs},D^\mathsf{c}_N)
\end{align}
holds for every $q\in (a,b)\cap(p_c,1]$, where $MS^N$ denotes the set of all finite collections of pairwise disjoint separating components $S$, and $c(S)$ denotes the number of separating components of $S$. 
\end{lemma}

\Lr{inc-ex lemma} will follow easily from the next lemma. We will use the notation $MS^N_n$ to denote the set of those finite collections of pairwise disjoint separating components $\{S_1,S_2,\ldots,S_k\}$ such that $|S_1|+|S_2|+\ldots+|S_k|=n$. The superscript reminds us of the dependence of the boxes on $N$. 

\begin{lemma}\label{exp dec}
For every $p>p_c$, there are $N=N(p)>0$, $t=t(p)>0$ and an interval $(a,b)$ containing $p$ such that 
\begin{align}\label{ineq}
\sum_{S\in MS^N_n} \Pr_q(\text{S is bad})\leq e^{-tn}
\end{align}
for every $n\geq 1$ and every $q\in (a,b)\cap (p_c,1]$. 
\end{lemma}
\begin{proof}
To prove the desired exponential decay we will use a standard renormalization technique  with a few modifications. We will first prove the exponential decay when $q=p$, and then we will use a continuity argument to obtain the desired assertion.

We will first show that there exists a constant $k>0$ depending only on $d$ such that for every $S\in MS^N_n$ we have $$\Pr_p(\text{S is bad})\leq c^{n/k},$$ where $c:=\Pr_p(B(o) \text{ is bad})$. Indeed, it is not hard to see that there is a constant $k=k(d)>0$ such that for every $S\in MS^N_n$ there is a subset $Y$ of $S$ of size at least $n/k$, all boxes of which are pairwise disjoint. As each box of $Y$ is bad whenever $S$ occurs, we have $$\Pr_p(\text{S is bad})\leq \Pr_p(\text{Y is bad}).$$
By independence $\Pr_p(\text{Y is bad})=c^{n/k}$ and the assertion follows.

We will now find an exponential upper bound for the number of elements of $S\in MS^N_n$. Since $N\Ls$ is isomorphic to $\Ls$, there is a constant $\mu>0$ depending only on $d$ and not on $N$, such that the number of connected subgraphs of $N\Ls$ with $n$ vertices containing a given vertex is at most $\mu^n$. However, an element of $MS^N_n$ might contain multiple separating components, and there are in general several possibilities for the reference vertices that each of them contains. To remedy this, consider one of the $d$ axis $X=(-x_1,x_0=B(o),x_1)$ of $N\Ls$ that contain the box $B(o)$, and let $X^+$, $X^-$ be its two infinite subpaths starting from $B(o)$. We will first show that any separating component of size $n$ contains one of the first $n$ elements of $X^+$. Indeed, consider an occurring separating component $S$ of size $n$, and notice that $S$ has to contain some vertex $x^+$ of $X^+$, and some vertex $x^-$ of $X^-$. The graph distance between $x^+$ and $x^-$ is at most $n$, as there is a path in $S$ connecting them. This implies that $x^+$ is one of the first $n$ elements of $X^+$, as desired.

Consider now a constant $M>0$ such that $m\mu^m\leq M^m$ for every integer $m\geq 1$. Consider also a partition $\{m_1,m_2,\ldots,m_k\}$ of $n$. It follows that the number of collections $\{S_1,S_2,\ldots,S_k\}$ with $|S_i|=m_i$ is at most $m_1m_2\ldots m_k \mu^n\leq M^n$, since we have at most $m_i\mu^{m_i}$ choices for each $S_i$. A well known result of Hardy \& Ramanujan \cite{HarRam} implies that the number of partitions of $n$ is at most $r^{\sqrt{n}}$ for some constant $r>0$. We can now deduce that the size of $MS^N_n$ is at most $r^{\sqrt{n}}M^n$, implying that 
$$\sum_{S\in MS^N_n} \Pr_p(\text{S is bad})\leq r^{\sqrt{n}}M^n c^{n/k}.$$
Notice that in the right hand side of the above inequality only $c$ depends on $N$. It is a standard result that $c$ converges to $0$ as $N$ tends to infinity \cite[Theorem 7.61]{Grimmett}. Choosing $N$ large enough so that $Mc^{1/k}<1$, we obtain the desired exponential decay.

Now notice that $c(q)=\Pr_q(B(o) \text{ is bad})$ is a polynomial in $q$, hence a continuous function, since it depends only on the state of the edges inside $B(o)$. This implies that we can choose an interval $(a,b)$ containing $p$ such that $Mc(q)^{1/k}<1$ for every $q\in (a,b)\cap (p_c,1]$. This completes the proof.
\end{proof}

\Lr{inc-ex lemma} follows now easily:

\begin{proof}[Proof of \Lr{inc-ex lemma}]
\Lr{exp dec} shows that $\sum_{S\in MS^N} \Pr_q(\text{S occurs}, D^\mathsf{c}_N)$ is finite, hence only finitely many separating components occur in almost any percolation configuration $\omega$ by the Borel-Cantelli lemma. Now the standard inclusion-exclusion principle implies that $$\mathbbm{1}_{\{\text{S occurs},D^\mathsf{c}_N\}}=\sum_{S\in MS^N}(-1)^{c(S)+1} \mathbbm{1}_{\{\text{S occurs},D^\mathsf{c}_N\}}.$$ Taking expectations we obtain 
$$\Pr_q(\text{S occurs}, D^\mathsf{c}_N)=\Ex_q(\sum_{S\in MS^N}(-1)^{c(S)+1} \mathbbm{1}_{\{\text{S occurs},D^\mathsf{c}_N\}}).$$
Since $$\Ex_q(\sum_{S\in MS^N} \mathbbm{1}_{\{\text{S occurs},D^\mathsf{c}_N\}})=\sum_{S\in MS^N} \Pr_q(\text{S occurs}, D^\mathsf{c}_N)$$ and the latter sum is finite, Fubini's theorem implies that 
$$\Ex_q(\sum_{S\in MS^N}(-1)^{c(S)+1} \mathbbm{1}_{\{\text{S occurs},D^\mathsf{c}_N\}})=\sum_{S\in MS^N}(-1)^{c(S)+1} \Pr_q(\text{S occurs},D^\mathsf{c}_N).$$ 
The proof is complete.
\end{proof}

We are now ready to prove \Tr{analytic}.

\begin{proof}[Proof of \Tr{analytic}]
Consider some $p\in (p_c,1]$. Let $N,t>0$, and the interval $(a,b)$ containing $p$, be as in \Lr{exp dec}. Then the expression $$1-\theta(q)=\Pr_q(D_N)+\sum_{n=1}^\infty \sum_{S\in MS^N_n} (-1)^{c(S)+1} \Pr_q(\text{S occurs},D^\mathsf{c}_N)$$ holds for every $q\in (a,b)\cap (p_c,1]$, and furthermore $$\Bigl\lvert\sum_{S\in MS^N_n} (-1)^{c(S)+1} \Pr_q(\text{S occurs},D^\mathsf{c}_N)\Bigr\rvert\leq e^{-tn}$$ for every $q\in (a,b)\cap (p_c,1]$.
The probability $\Pr_q(D_N)$ is a polynomial in $q$, hence analytic, because it depends on finitely many edges. Moreover, the event $\{\text{S occurs},D^\mathsf{c}_N\}$ depends only on the state of the edges lying in $S\cup\partial_\sboxt S$ and the box $B(o,N)$. The number of edges of each box is $O(N^d)$, hence the event $\{\text{S occurs},D^\mathsf{c}_N\}$ depends only on $O(N^d n)$ edges. The desired assertion follows now from \Tr{cor general}.
\end{proof}

\subsection{Exponential tail of $\partial^b \mathcal{S}_o$} \label{Pete} 

\Lr{exp dec} easily implies that the size of $\partial^b \mathcal{S}_o$, as defined in the proof of \Lr{Co finite}, has an exponential tail:

\begin{theorem}\label{bad boundary}
For every $p>p_c$, there are constants $N=N(p)>0$ and $t=t(p)>0$ such that 
$$\Pr_p(|\partial^b \mathcal{S}_o|\geq n)\leq e^{-tn}$$ for every $n\geq 1$.
\end{theorem}
\begin{proof}
Assume that $|\partial^b \mathcal{S}_o|\geq n$, and consider the separating component $S$ associated to $C_o$. Then the boxes of $S\cup \partial_\sboxt S$ must contain at least $n$ edges. Hence the number of boxes of $S\cup \partial_\sboxt S$ is at least $cn/N^d$ for some constant $c>0$. Moreover, we have $|\partial_\sboxt S|\leq (3^d-1)|S|$, because each box of $\partial_\sboxt S$ has at least one neighbour in $S$, and each box in $S$ has at most $3^d-1$ neighbours. Therefore, $S$ contains at least $cn/(3N)^d$ boxes. The desired assertion follows from \Lr{exp dec}.
\end{proof}

We recall that for every $p\in (p_c,1-p_c]$, the probability $\Pr_p(|\partial C_o|\geq n)$ does not decay exponentially in $n$ \cite{KeZhaPro,ExpGrowth}. This implies that for those values of $p$, $\partial^b \mathcal{S}_o$ has typically smaller order of magnitude than the standard minimal edge cut of $C_o$.

\medskip
As a corollary, we re-obtain a result of Pete \cite{Pete} which states that when $C_o$ is finite, the number of touching edges between $C_o$ and the unique infinite cluster, which we denote $C_{\infty}$, has an exponential tail. A \defi{touching edge} is an edge in $\partial^E C_o\cap \partial^E C_{\infty}$. We denote the number of (closed) touching edges joining $C_o$ to the infinite component $C_{\infty}$ by $\phi(C_o,C_{\infty})$.

\begin{corollary}\label{touching}
For every $p>p_c$, there is some $c=c(p,d)>0$ such that
$$\Pr_p(|C_o|<\infty,\phi(C_o,C_{\infty})\geq t)\leq e^{-ct}$$
for every $t\geq 1$.
\end{corollary}
\begin{proof}
The result follows from \Tr{bad boundary} by observing that $C_{\infty}$ has to lie in the unbounded component of $\mathbb{L}^d\setminus \partial^b \mathcal{S}_o$, hence all relevant edges belong to $\partial^b \mathcal{S}_o$.
\end{proof}
 
\section{Analyticity of $\tau$}\label{section tau}

In the previous section we proved that $\theta$ is analytic above $p_c$ for every $d\geq 3$. Some further challenges arise when one tries to prove that other functions describing the macroscopic behaviour of our model are analytic functions of $p$. The main obstacle is that events of the form $\{x \text{ is connected to }y\}$ are not fully determined, in general, by the configuration inside $S\cup \partial_\sboxt S$. In this section we show how one can remedy this issue, and we will prove that the $k$-point function $\tau$ and its truncated version $\tau^f$ are analytic functions above $p_c$ for every $d\geq 3$. We will then deduce that the truncated susceptibility $\mathbb{E}(|C_o|; |C_o|<\infty)$ and the free energy $\mathbb{E}(|C_o|^{-1})$ are analytic functions as well.

Given a $k$-tuple $\vx=\{x_1,\ldots,x_k\}, k\geq 2$ of vertices of $\Z^d$, the function $\tau_{\vx}(p)$ denotes the probability that $\vx$ is contained in a cluster of Bernoulli percolation on $\Z^d$ with parameter $p$. Similarly, $\tau^f_{\vx}(p)$ denotes the probability that $\vx$ is contained in a {\it finite} cluster. We will write $MS^N(\vx)$ for the set of all finite collections of separating components surrounding some vertex of $\vx$, and $MS^N_n(\vx)$ for the set of those elements of $MS^N(\vx)$ that have size $n$.

Arguing as in the proof of \Lr{exp dec} we obtain the following: 
\begin{lemma}\label{exp dec x}
For every $p>p_c$, there are $N=N(p)>0$, $t=t(p)>0$, and an interval $(a,b)$ containing $p$, such that 
\begin{align}\label{ineq x}
\sum_{S\in MS^N_n} \Pr_q(\text{S occurs})\leq e^{-tn}
\end{align}
for every $n\geq 1$ and every $q\in (a,b)\cap (p_c,1]$. 
\end{lemma}

We are now ready to prove that $\tau$ and $\tau^f$ are analytic. 

\begin{theorem}\label{tau}
For every $d\geq 3$ and every finite set $\vx$ of vertices of $\Z^d$, the functions $\tau_{\vx}(p)$ and $\tau^f_{\vx}(p)$ admit analytic extensions to a domain of $\C$ that contains the interval $(p_c, 1]$.

Moreover, for every $p\in (p_c, 1]$ and every finite set $\vx$ such that $diam(\vx)\geq N/5$, there is a closed disk $D(p,\delta),\delta>0$ and positive constants $c_1=c_1(p,\delta),c_2=c_2(p,\delta)$ such that $$|\tau^f_{\vx}(z)|\leq c_1 e^{-c_2{diam(\vx)}}$$
\fe\ $z\in D(p,\delta)$ for such an analytic extension $\tau^f_{\vx}(z)$ of $\tau^f_{\vx}(p)$.
\end{theorem}
\begin{proof}
We start by showing that $\tau^f_{\vx}(p)$ is analytic. Suppose $\vx=\{x_1,\ldots,x_k\}$, and let $A$ be the event that $diam(C_{x_i})\geq N/5$ for every $i\leq k$. We will write $\{\vx \text{ is connected}\}$ to denote the event that all vertices of $\vx$ belong to the same cluster, which we denote $C_{\vx}$. When $C_{\vx}$ is finite and both events $\{\vx \text{ is connected}\}$ and $A$ occur, we will write $\mathcal{S}_{\vx}$ for the separating component of the latter cluster, namely the $N\Ls$-component of $\partial C_{\vx}(N)$. The event $\{\text{S occurs}\}$ is defined as in the previous section except that now $C_o$ is replaced by $C_{\vx}$, i.e. the event $\{\vx \text{ is connected}\}$ occurs in a witness $\omega'$, and $S$ contains $\partial C_{\vx}(\omega')(N)$.  With the above definitions we have $$\tau^f_{\vx}(p)=\Pr_p(A^\mathsf{c}, \vx \text{ is connected})+\sum_S \Pr_p(A,\vx \text{ is connected}, \mathcal{S}_{\vx}=S),$$
where the sum ranges over all possible separating components separating all of $\vx$ from infinity. 

Our aim is to further decompose the events of the above expansion into simpler ones that we have better control of, and then use the inclusion-exclusion principle. We will first introduce some notation. Given a separating component $S$ as above, we first decompose $\vx$ into two sets $\vx_{out}$ and $\vx_{in}$, where $\vx_{out}$ denotes the set of those vertices of $\vx$ lying in some finite component of $N\Ls\setminus (S\cup \partial_\sboxt S)$, and $\vx_{in}:= \vx \backslash \vx_{out}$ its complement. We write $\{\vx \rightarrow S\}$ for the event that no separating component separating some $x_i\in \vx$ from $S$ occurs; to be more precise, the event $\{\vx \rightarrow S\}$ means  that for each $x_i\in \vx_{out}$, no separating component that surrounds $x_i$ and lies entirely in some of the finite components of $N\Ls\setminus (S\cup \partial_\sboxt S)$ occurs. 

Consider now some vertex $x$ in $\vx_{out}$, and let $F$ be the component of $\partial_\sboxt S$ that separates $x$ from $S$. We claim that when $S$ and the events $A$, $\{\vx \rightarrow S\}$ all occur, then $x$ is connected to the unique large cluster of $F$. In particular, if another vertex of $\vx$ lies in the same finite component of $N\Ls\setminus (S\cup \partial_\sboxt S)$ as $x$ does, then both vertices are connected to the unique large cluster of $F$, hence they are connected to each other. To see that the claim holds, notice that $C_{x}$ has to be finite, because $S\cup\partial_\sboxt S$ contains a minimal edge cut of closed edges that surrounds all vertices of $\vx$, hence $x$. Now $\partial C_{x}(N)$ has to intersect $S$, because it cannot lie entirely in $N\Ls\setminus (S\cup \partial_\sboxt S)$ by our assumption. This implies that $x$ is connected to some vertex inside $S$, hence it must first visit the unique large cluster of $F$, as desired.

We now define $\mathcal{C}$ to be the event that 
\begin{itemize}
\item all vertices of $\vx_{in}$ are connected to each other with open paths lying in $S\cup \partial_\sboxt S$,
\item the unique large percolation clusters of the components $F$ of $\partial_\sboxt S$ that separate some $x_i\in \vx_{out}$ from $S$ are connected to each other with open paths lying $S\cup \partial_\sboxt S$,
\item all vertices of $\vx_{in}$ are connected to all such percolation clusters with open paths lying in $S\cup \partial_\sboxt S$.
\end{itemize}
(It is possible that either $\vx_{in}$ or $\vx_{out}$ is the empty set, in which case the third item and one of the first two are empty statements.)
We claim that when $S$ and the events $A$, $\{\vx \rightarrow S\}$ and $\{\vx \text{ is connected}\}$ all occur, then the event $\mathcal{C}$ occurs as well. Indeed, consider a vertex $x\in\vx_{out}$, and let $F$ be the component of $\partial_\sboxt S$ that separates $x$ from $S$, as above. Any open path connecting $x$ to some vertex of $\vx_{out}$ which does not lie in the same finite component of $N\mathbb{L}^d_\sboxt$ that $x$ does, has to first visit the unique large percolation cluster of $F$. Hence it suffices to prove that when two vertices $x_i$ and $x_j$ of $\vx_{in}$ lie in the same cluster, there is always an open path connecting them lying entirely in $S\cup \partial_\sboxt S$.
To this end, assume that there is a path $P$ in $\omega$ connecting $x_i$ to $x_j$, which does not lie entirely in $S\cup \partial_\sboxt S$. Arguing as in the proof of \Lr{finite}, we can modify $P$ to obtain an open path $P'$ connecting $x_i$ to $x_j$ which lies entirely in $S\cup \partial_\sboxt S$. The desired claim follows now easily.

Combining the above claims, we conclude that the events $\{A,\vx \text{ is connected},$ $\mathcal{S}_{\vx}=S\}$ and $\{A,\mathcal{C},\vx\rightarrow S,S \text{ occurs}\}$ coincide, and thus $$\Pr_p(A,\vx \text{ is connected}, \mathcal{S}_{\vx}=S)=\Pr_p(A,\mathcal{C},\vx\rightarrow S,S \text{ occurs}).$$
Using the inclusion-exclusion principle we obtain that 
\begin{align}
\begin{split}
\Pr_p(A,\mathcal{C},\vx\rightarrow S,S \text{ occurs})=\Pr_p(A,\mathcal{C},S \text{ occurs})+ \\ \sum_{T}(-1)^{c(T)}\Pr_p(A,T \text{ occurs},\mathcal{C},S \text{ occurs}),
\end{split}
\end{align}
where the latter sum ranges over all finite collections $T$ of separating components separating $\vx$ from $S$. Collecting now all the terms 
we obtain that 
\begin{align}
\begin{split}
\tau^f_{\vx}(p)=\Pr_p(A^\mathsf{c},\vx \text{ is connected})+\\ \sum_S\Big(\Pr_p(A,\mathcal{C},S \text{ occurs})+\sum_{T}(-1)^{c(T)}\Pr_p(A,T \text{ occurs},\mathcal{C},S \text{ occurs})\Big).
\end{split}
\end{align}

Notice that by combining $S$ and $T$ we obtain an element of $MS^N(\vx)$, hence we can use \Lr{exp dec x}, and then argue as in the proof of \Tr{analytic} to obtain that $\tau^f_{\vx}$ is analytic above $p_c$.

\bigskip
We will now prove the analyticity of $\tau_{\vx}$.  Since $\tau^f_{\vx}$ is analytic, it suffices to prove that $\tau_{\vx}-\tau^f_{\vx}$ is analytic.
It is well-known that the infinite cluster is unique in our setup \cite{BKunique}, and this implies that 
$\tau_{\vx}-\tau^f_{\vx} = \Pr(|C_{x_1}|=\infty,\ldots,|C_{x_k}|=\infty)$. The  latter probability is complementary to $\Pr(\cup_{i=1}^k \{|C_{x_i}|<\infty\})$, which is in turn equal to
$$\Pr(\cup_{i=1}^k \{|C_{x_i}|<\infty\})=\Pr(A^\mathsf{c})+\Pr(\big(\cup_{i=1}^k \{|C_{x_i}|<\infty\}\big)\cap A).$$
Define the event $\{\text{S occurs for some } x_i\in \vx\}$ as in the previous section expect that now we require the existence of a witness $\omega'$ such that $S$ contains $\partial C_{x_i}(\omega')(N)$ for some $x_i\in \vx$.
We can expand the latter term as an infinite sum using the inclusion-exclusion principle to obtain
$$\Pr(\big(\cup_{i=1}^k \{|C_{x_i}|<\infty\}\big)\cap A)=\sum (-1)^{c(S)+1}\Pr(\text{S occurs for some } x_i\in \vx, A),$$
where now we require our separating components to surround some $x_i\in \vx$.
Arguing as in the proof of \Tr{analytic} we obtain that $\tau_{\vx}-\tau^f_{\vx}$ is analytic, as desired.

\bigskip
For the second claim of the theorem, notice that when $diam(\vx)\geq N/5$, the probability $\Pr(A^\mathsf{c}, \vx \text{ is connected})$ is equal to $0$. Hence our expansion for $\tau^f_{\vx}$ simplifies to 
$$\tau^f_{\vx}(p)=\sum_S\Big(\Pr(A,\mathcal{C},S \text{ occurs})+\sum_{T}(-1)^{c(T)}\Pr(A,T \text{ occurs},\mathcal{C},S \text{ occurs})\Big).$$
Our goal is to show that for every $p>p_c$ there are some constants $\delta,t>0$  such that 
\begin{align}\label{abs}
\Bigl\lvert\sum_{|S|=n}\Big(\Pr_p(A,\mathcal{C},S \text{ occurs})+\sum_{T}(-1)^{c(T)}\Pr_p(A,T \text{ occurs},\mathcal{C},S \text{ occurs})\Big)\Bigr\rvert \leq e^{-tn}
\end{align}
for every $z\in D(p,\delta)$ for the analytic extensions of the above probabilities. Then the desired claim will follow easily from the observation that any plausible separating component $S$ of $\vx$ must have size $\Omega(diam(\vx))$. 

Notice that the event $A$ depends only on the edges in the boxes $B(x_i)$, $x_i\in \vx$. Moreover, the events $\mathcal{C}$ and 
$\{S \text{ occurs}\}$ depend on $O(|S|)$ edges, while the event $\{T \text{ occurs}\}$ depends on $O(|T|)$ edges. We can now use \Lr{C equals S NN} to conclude that there is a constant $c=c(p,\delta,N)>1$ (perhaps slightly larger than that of \Lr{C equals S NN}) such that
$$|\Pr_z(A,\mathcal{C},S \text{ occurs})|\leq c^{|S|} \Pr_{p'}(A,\mathcal{C},S \text{ occurs})$$
and $$|\Pr_z(A,T \text{ occurs},\mathcal{C},S \text{ occurs})|\leq c^{|S|+|T|} \Pr_{p'}(A,T \text{ occurs},\mathcal{C},S \text{ occurs})$$
for every $z\in D(p,\delta)$, where $p'=p+\delta$ if $p<1$, and $p'=1-\delta$ if $p=1$. Moreover, we can always choose $c$ in such a way that $c\rightarrow 1$ as $\delta\rightarrow 0$. Hence we have 
\begin{align}
\begin{split}
\Bigl\lvert\sum_{|S|=n}\Big(\Pr_z(A,\mathcal{C},S \text{ occurs})+\sum_{T}(-1)^{c(T)}\Pr_z(A,T \text{ occurs},\mathcal{C},S \text{ occurs})\Big)\Bigr\rvert \leq \\
c^n\sum_{|S|=n}\Big(\Pr_{p'}(A,\mathcal{C},S \text{ occurs})+\sum_{T}c^{|T|}\Pr_{p'}(A,T \text{ occurs},\mathcal{C},S \text{ occurs})\Big)
\end{split}
\end{align}
by the triangle inequality. It follows from \Lr{exp dec x} that the sum 
$$\sum_{|S|=n}\Big(\Pr_{p'}(A,\mathcal{C},S \text{ occurs})+\sum_{T}\Pr_{p'}(A,T \text{ occurs},\mathcal{C},S \text{ occurs})\Big)$$ decays exponentially in $n$, and by choosing $\delta$ small enough we can ensure that 
$$c^n\sum_{|S|=n}\Big(\Pr_{p'}(A,\mathcal{C},S \text{ occurs})+\sum_{T}c^{|T|}\Pr_{p'}(A,T \text{ occurs},\mathcal{C},S \text{ occurs})\Big)$$ decays exponentially in $n$ as well, hence
\eqref{abs} holds.
The proof is now complete.
\end{proof}

Using \Tr{tau} we can now prove the following results.

\begin{theorem}\label{expectation}
For every $k\geq 1$ and every $d\geq 3$, the functions $\chi_k^f(p):=\mathbb{E}_p(|C(o)|^k; |C(o)|<\infty)$ are analytic in $p$ on the interval $(p_c,1]$. 
\end{theorem}
\begin{proof}
Let us show that $\chi^f(p):=\mathbb{E}(|C(o)|; |C(o)|<\infty)$ is analytic. The case $k\geq 2$ will follow similarly.
We observe that, by the definitions,
$$\chi^f(p)=\sum_{x\in \Z^d} \tau^f_{\{o,x\}}=1+\sum_{x\in \Z^d\setminus \{o\}} 
\tau^f_{\{o,x\}}.$$ The probabilities $\tau^f_{\{o,x\}}$ admit analytic extensions by \Tr{tau}, and so it suffices to prove that the 
sum $\sum_{x\in \Z^d\setminus \{o\}} \tau^f_{\{o,x\}}$ converges uniformly on an open neighbourhood of $(p_c,1]$. This follows easily from the estimates of the second sentence of \Tr{tau}, and the polynomial growth of $\Z^d$.
\end{proof}

\begin{theorem}\label{free energy}
For every $d\geq 3$, the free energy $\kappa=\mathbb{E}(|C_o|^{-1})$ is analytic in $p$ on the interval $(p_c,1]$. 
\end{theorem}
\begin{proof}
It is  known \cite{Uniqueness} that $\kappa$ is differentiable on $(p_c,1)$ with derivative equal to $$f(p):=\dfrac{1}{2(1-p)}\sum_{x\in N(o)} \big(1-\tau_{\{o,x\}}(p)\big).$$ Since each $\tau_{\{o,x\}}$ is analytic on the interval $(p_c,1]$, and $\tau_{\{o,x\}}(1)=1$, $f$ is analytic on $(p_c,1]$ as well. So far we know that $\kappa$ coincides with a primitive $F$ of $f$ only on  $(p_c,1)$, which implies that $\kappa$ is analytic on that interval. In fact, $\kappa$ coincides with $F$ on the whole interval $(p_c,1]$. Indeed, we simply need to verify that $\kappa$ is continuous from the left at $1$. To see this notice that $\kappa(1)=1-\theta(1)=0$ and $\kappa(p)\leq 1-\theta(p)$. Since $\theta$ is continuous from the left at $1$, which follows e.g. by \Tr{analytic}, we have that $\kappa$ is continuous from the left at $1$ as well, hence coincides with $F$ on the whole interval $(p_c,1]$. It now follows that $\kappa$ is analytic in $p$ on the interval $(p_c,1]$, as desired. 
\end{proof}

\acknowledgement{We thank Tom Hutchcroft for suggesting using the results of \cite{KeZhaPro}.}

\bibliographystyle{plain}
\bibliography{collective}

\begin{thebibliography}{10}

\bibitem{AB}
M.~Aizenman and D.~J. Barsky.
\newblock {Sharpness of the phase transition in percolation models}.
\newblock {\em Communications in Mathematical Physics}, 108:489--526, 1987.

\bibitem{AiDeSoLow}
M.~{Aizenman}, F.~{Delyon}, and B.~{Souillard}.
\newblock {Lower bounds on the cluster size distribution}.
\newblock {\em Journal of Statistical Physics}, 23(3):267--280, 1980.

\bibitem{Uniqueness}
M.~Aizenman and H.~Kesten C.~M. Newman.
\newblock Uniqueness of the infinite cluster and continuity of connectivity
  functions for short and long range percolation.
\newblock {\em Communications in Mathematical Physics}, 111(4):505--531, 1987.

\bibitem{BoRioShor}
B.~Bollob{\'a}s and O.~Riordan.
\newblock {A short proof of the Harris--Kesten theorem}.
\newblock {\em Bulletin of the London Mathematical Society}, 38(3):470--484,
  2006.

\bibitem{BrPrSaSco}
G.~A. Braga, A.~Procacci, R.~Sanchis, and B.~Scoppola.
\newblock {Percolation Connectivity in the Highly Supercritical Regime}.
\newblock {\em Markov Processes And Related Fields}, 10(4):607--628, 2004.

\bibitem{BrPrSa}
G.~A. Braga, A.~Proccaci, and R.~Sanchis.
\newblock {Analyticity of the $d$-Dimensional Bond Percolation Probability
  Around $p=1$}.
\newblock {\em Journal of Statistical Physics}, 107:1267--1282, 2002.

\bibitem{BKunique}
R.~M. Burton and M.~Keane.
\newblock {Density and uniqueness in percolation}.
\newblock {\em Communications in Mathematical Physics}, 121(3):501--505, 1989.

\bibitem{ChChNeBer}
J.~T. Chayes, L.~Chayes, and C.~M. Newman.
\newblock {Bernoulli percolation above threshold: an invasion percolation
  analysis}.
\newblock {\em The Annals of Probability}, pages 1272--1287, 1987.

\bibitem{analyticity}
A.~Georgakopoulos and C.~Panagiotis.
\newblock {Analyticity results in Bernoulli percolation}.
\newblock {\em {arXiv:1811.07404}}.

\bibitem{ExpGrowth}
A.~Georgakopoulos and C.~Panagiotis.
\newblock On the exponential growth rates of lattice animals and interfaces,
  and new bounds on $p_c$.
\newblock {\em arXiv:1908.03426}.

\bibitem{Griffiths}
R.~Griffiths.
\newblock {Nonanalytic behavior above the critical point in a random Ising
  ferromagnet}.
\newblock {\em Physical Review Letters}, 23(1):17, 1969.

\bibitem{GriDif}
G.~Grimmett.
\newblock On the differentiability of the number of clusters per vertex in the
  percolation model.
\newblock {\em Journal of the London Mathematical Society}, 2(2):372--384,
  1981.

\bibitem{GrimmettDisordered}
G.~Grimmett.
\newblock Percolation and disordered systems.
\newblock In {\em Lectures on probability theory and statistics}, pages
  153--300. 1997.

\bibitem{GriMar}
G.~Grimmett and J.~Marstrand.
\newblock The supercritical phase of percolation is well behaved.
\newblock {\em Proceedings of the Royal Society of London. Series A:
  Mathematical and Physical Sciences}, 430(1879):439--457, 1990.

\bibitem{Grimmett}
Geoffrey Grimmett.
\newblock {\em {Percolation, \emph{{Second Edition}}}}.
\newblock Grundlehren der mathematischen Wissenschaften. Springer, 1999.

\bibitem{HarRam}
G.~H. Hardy and S.~Ramanujan.
\newblock {Asymptotic Formulae in Combinatory Analysis}.
\newblock {\em Proceedings of the London Mathematical Society}, 17:75--115,
  1918.

\bibitem{SitePercoPlane}
J.~Haslegrave and C.~Panagiotis.
\newblock Site percolation and isoperimetric inequalities for plane graphs.
\newblock {\em {arXiv:1905.09723}}.

\bibitem{HerHutSup}
{J.~Hermon and T.~Hutchcroft}.
\newblock {Supercritical percolation on nonamenable graphs: Isoperimetry,
  analyticity, and exponential decay of the cluster size distribution}.
\newblock {\em {arXiv:1904.10448}}.

\bibitem{KLM}
W.~Kager, M.~Lis, and R.~Meester.
\newblock {The signed loop approach to the Ising model: foundations and
  critical point}.
\newblock {\em Journal of Statistical Physics}, 152(2):353--387, 2013.

\bibitem{KestenCritical}
H.~Kesten.
\newblock {The critical probability of bond percolation on the square lattice
  equals 1/2}.
\newblock {\em Communications in Mathematical Physics}, 74(1):41--59, 1980.

\bibitem{Ke81}
H.~Kesten.
\newblock {Analyticity Properties and Power Law Estimates of Functions in
  Percolation Theory}.
\newblock {\em Journal of Statistical Physics}, 25(4):717--756, 1981.

\bibitem{KeZhaPro}
H.~Kesten and Y.~Zhang.
\newblock {The probability of a large finite cluster in supercritical Bernoulli
  percolation}.
\newblock {\em Annals of Probability}, 18(2):537--555, 1990.

\bibitem{KestenBook}
Harry Kesten.
\newblock {\em {Percolation theory for mathematicians}}.
\newblock Springer, 1982.

\bibitem{KunSou}
H.~Kunz and B.~Souillard.
\newblock {Essential Singularity in Percolation Problems and Asymptotic
  Behavior of Cluster Size Distribution}.
\newblock {\em Journal of Statistical Physics}, 19(1):77--106, 1978.

\bibitem{LyonsBook}
R.~Lyons and Y.~Peres.
\newblock {\em {Probability on trees and networks}}.
\newblock Cambridge University Press, 2016.

\bibitem{MenCoi}
M.~V. Menshikov.
\newblock {Coincidence of critical-points in the percolation problems}.
\newblock {\em Doklady akademii nauk sssr}, 288(6):1308--1311, 1986.

\bibitem{Onsager}
L.~Onsager.
\newblock {Crystal statistics. I. A two-dimensional model with an
  order-disorder transition}.
\newblock {\em Physical Review}, 65(3-4):117, 1944.

\bibitem{Pete}
G.~Pete.
\newblock {A note on percolation on $\mathbb{Z}^d$: Isoperimetric profile via
  exponential cluster repulsion}.
\newblock {\em Electronic Communications in Probability}, 13:377--392, 2008.

\bibitem{SykesEssam}
M.~F. Sykes and J.~W Essam.
\newblock Exact critical percolation probabilities for site and bond problems
  in two dimensions.
\newblock {\em Journal of Mathematical Physics}, 5(8):1117--1127, 1964.

\bibitem{TimCut}
A.~Tim{\'a}r.
\newblock {Cutsets in Infinite Graphs}.
\newblock {\em Combinatorics, Probability and Computing}, 16:159--166, 2007.

\bibitem{EntGri}
A.~C.~D. van Enter.
\newblock {Griffiths Singularities}.
\newblock In {\em Modern Encyclopedia of Mathematical Physics}, 2007.

\bibitem{Complete}
A.~C.~D. van Enter, R.~Fern{\'a}ndez, R.~H. Schonmann, and S.~B. Shlosman.
\newblock {Complete analyticity of the 2D Potts model above the critical
  temperature}.
\newblock {\em Communications in mathematical physics}, 189(2):373--393, 1997.

\end{thebibliography}
\end{document}